\documentclass[12pt]{amsart}
\usepackage{latexsym,amsmath,amsfonts,amssymb,amsthm}
\def\E{\mathbb E}
\def\P{\mathbb P}

\def\vol{{\rm{vol}}}

\theoremstyle{plain}
\newtheorem{Thm}{Theorem}
\newtheorem{Lem}{Lemma}

\theoremstyle{definition}

\theoremstyle{remark}

\begin{document}
\title{Inverse Erd\H os-Fuchs theorem for $k$-fold sumsets}
\author{Li-Xia Dai}
\email{lilidainjnu@163.com}
\address{School of Mathematical Sciences, Nanjing Normal University, Nanjing 210023, People's Republic of China}
\author{Hao Pan}
\email{haopan79@yahoo.com.cn}
\address{Department of Mathematics, Nanjing University,
Nanjing 210093, People's Republic of China}
\begin{abstract}
We generalize a result of Ruzsa on the inverse Erd\H os-Fuchs theorem for $k$-fold sumsets.
\end{abstract}

\maketitle

\section{Introduction}
\setcounter{equation}{0} \setcounter{Thm}{0} \setcounter{Lem}{0}
\setcounter{Cor}{0}

For $r>0$, let $N(r)$ count the number of lattice points inside
the boundary of a circle with the center at the origin and radius
$r$. The famous Gauss circle problem says that
$$
N(r)=\pi r^2+O(r^{1/2+\epsilon}).
$$
However, the current best result is that here
$O(n^{1/2+\epsilon})$ can be  replaced by $O(n^{131/208})$. On the
other hand, using the techniques of the Fourier analysis, Hardy
proved that
$$
N(r)=\pi r^2+O(r^{1/2}(\log r)^{1/4})
$$
can't hold for all sufficiently large $r$.

Consider a sequence $A=\{a_1\leq a_2\leq a_3\leq\ldots\}$ of non-negative integers with $\lim_{n\to\infty} a_n=\infty$.
For a positive integer $n$, define
$$
r_{kA}(n)=|\{(i_1,i_2,\ldots,i_k):\,a_{i_1}+a_{i_2}+\cdots+a_{i_k}=n,\
a_{i_1},\ldots,a_{i_k}\in A\}|,
$$
i.e., $r_{kA}(n)$ counts the number of representations of  $n$ as
the sum of $k$ elements in $A$. It is easy to see that
$$
N(r)=\sum_{m\leq r^2}r_{2A}(m)
$$
provided $A=\{0,1,4,9,16,\ldots\}$.

In \cite{EF56}, Erd\H os and Fuchs proved that for any sequence $A$ and constant $C>0$,
\begin{equation}\label{ef}
\sum_{m\leq n}r_{2A}(m)=Cn+o(n^{1/4}(\log n)^{-1/2})
\end{equation}
can't hold for all sufficiently large $n$. Although here
$o(n^{1/4}(\log n)^{-1/2})$ is slightly weaker than Hardy's bound
$O(n^{1/4}(\log n)^{1/4})$, the Erd\H os and Fuchs theorem is
valid for any sequence $A$ of non-negative integers, rather than
only for $A=\{0,1,4,9,16,\ldots\}$. Subsequenty, Jurkat
(unpublished), and Montgomery and Vaughan  \cite{MV90} removed
$(\log n)^{-1/2}$ in (\ref{ef}). Nowadays, there are several
different generalizations of the Erd\H os-Fuchs theorem (cf.
\cite{CT11, H01, H02, H04, T09}). For example, Tang \cite{T09}
proved that for $k\geq 2$,
$$
\sum_{m\leq n}r_{kA}(m)=Cn+o(n^{1/4})
$$
can't hold for all sufficiently large $n$.

In the opposite direction, Vaughan asked whether there exists a sequence $A$ and $C>0$ such that
$$
\sum_{m\leq n}r_{2A}(m)=Cn+O(n^{1/4+\epsilon}).
$$
With help of a probabilistic discussion, Ruzsa \cite{R97} gave an affirmative answer to this question. In fact,
he proved that there exists a sequence $A=\{a_1\leq a_2\leq\ldots\}$ satisfying
$$
\sum_{m\leq n}r_{2A}(m)=Cn+O(n^{1/4}\log n)
$$
for all sufficiently large $n$.

It is natural to ask whether Ruzsa's result can be generalized to
the $k$-fold sums. In this note, we shall prove that
\begin{Thm}\label{t1} Suppose that $k\geq 2$ is an integer and $\beta<k$ is a positive real number. Then there exists a sequence $A=\{a_1\leq a_2\leq a_2\leq\cdots\}$ of positive integers, satisfying
$$
\sum_{m\leq n}r_{kA}(m)-Cn^\beta=\begin{cases}
O(n^{\beta-\beta(k+\beta)/k^2}\sqrt{\log n}),&\text{if }k>2\beta,\\
O(n^{\beta-3\beta/(2k)}\sqrt{\log n}),&\text{if }k<2\beta,\\
O(n^{\beta-3/4}\log n),&\text{if }k=2\beta,
\end{cases}
$$
where $C$ is a constant.
\end{Thm}
The proof of Theorem \ref{t1} will be given in the next two sections. And throughout this paper, the implied constants in $O$, $\ll$, $\gg$ only depend on $\beta$ and $k$.

\section{The probabilistic approach}
\setcounter{equation}{0} \setcounter{Thm}{0} \setcounter{Lem}{0}
\setcounter{Cor}{0}
Let $\theta_i$ be independent random variables which is uniformly distributed in the interval $[i,i+1]$. Let
$$
a_i=\lfloor \theta_i^\alpha\rfloor,
$$
where $\alpha=k/\beta$.
Clearly $a_i\leq a_j$ provided $i\leq j$. Now assume that $n$ is sufficiently large. Let
$$
\Omega=\{(x_1,x_2,\ldots,x_k):\,x_1^\alpha+\cdots+x_k^\alpha\leq n,\ x_i\geq 0\}.
$$
It is easy to see that the volume of $\Omega$
$$
\vol(\Omega)=C_{k,\alpha}n^{k/\alpha},
$$
where $C_{k,\alpha}$ is a constant only depending on $k$ and $\alpha$.
For $i_1,\ldots,i_k\geq 0$, let $d_{i_1,\ldots,i_k}$ denote the volume of the intersection of
$\Omega$ and the hypercube
$$
[i_1,i_1+1]\times[i_2,i_2+1]\times\cdots\times[i_k,i_k+1].
$$
It is easy to see that
$$
\sum_{i_1,\ldots,i_k\geq 0}d_{i_1,\ldots,i_k}=\vol(\Omega)=C_{k,\alpha}n^{\beta}.
$$
Define
$$
\delta_{i_1,\ldots,i_k}=\begin{cases}1&\text{if }\theta_{i_1}^\alpha+\cdots+\theta_{i_k}^\alpha\leq n,\\
0&\text{otherwise,}\end{cases}
$$
and let
$$
\sigma_n=\sum_{i_1,\ldots,i_k\geq 0}\delta_{i_1,\ldots,i_k}.
$$
Let
$$A=\{a_1,a_2,\ldots\}.$$
Since
$$
\lfloor \theta_{i_1}^\alpha\rfloor+\cdots+\lfloor \theta_{i_k}^\alpha\rfloor\leq \theta_{i_1}^\alpha+\cdots+\theta_{i_k}^\alpha\leq\lfloor \theta_{i_1}^\alpha\rfloor+\cdots+\lfloor \theta_{i_k}^\alpha\rfloor+k,
$$
clearly we have
$$
\sigma_n\leq\sum_{m\leq n}r_{kA}(m)\leq\sigma_{n+k}.
$$
Note that $\delta_{i_1,\ldots,i_k}=d_{i_1,\ldots,i_k}$ if $(i_1+1)^\alpha+\cdots+(i_k+1)^\alpha\leq n$ or $i_1^\alpha+\cdots+i_k^\alpha\geq n$. So we only need to consider those $i_1,\ldots,i_k$ such that
$$i_1^\alpha+\ldots+i_k^\alpha<n<(i_1+1)^\alpha+\ldots+(i_k+1)^\alpha.
$$
\begin{Lem}\label{l1} Suppose that $a_1,\ldots,a_k$ are positive integers. Then
$$
|\{(i_1,\ldots,i_k):\, a_1i_1^\alpha+\cdots+a_ki_k^\alpha<n<a_1(i_1+1)^\alpha+\cdots+a_k(i_k+1)^\alpha\}|=O(n^{(k-1)/\alpha}),
$$
where the implied constant in $O$ only depends on $a_1,\ldots,a_k$.
\end{Lem}
\begin{proof}
Consider
$$
\Omega_R^{(a_1,\ldots,a_k)}=\{(x_1,\ldots,x_k):\,a_1x_1^\alpha+\cdots+a_kx_k^\alpha\leq R^\alpha,\ x_i\geq 0\}.
$$
Clearly
$$
\vol(\Omega_R^{(a_1,\ldots,a_k)})=C_{k,\alpha}^{(a_1,\ldots,a_k)}R^k,
$$
where $C_{k,\alpha}^{(a_1,\ldots,a_k)}$ is a constant. In view of the differential mean value theorem, we have
$$
(x+1)^\alpha-x^\alpha\leq \alpha(x+1)^{\alpha-1}
$$
Hence
$$
a_1i_1^\alpha+\cdots+a_ki_k^\alpha<n<a_1(i_1+1)^\alpha+\cdots+a_k(i_k+1)^\alpha
$$
implies that
$$
n-(a_1+\cdots+a_k)(n^{1-1/\alpha}+1)<a_1i_1^\alpha+\ldots+a_ki_k^\alpha
$$
and
$$a_1(i_1+1)^\alpha+\cdots+a_k(i_k+1)^\alpha<n+(a_1+\cdots+a_k)(n^{1-1/\alpha}+1).
$$
That is, the hypercube $[i_1,i_1+1]\times\cdots\times[i_k,i_k+1]$ is completely contained in
$$
\Omega_{(n+(a_1+\cdots+a_k)(n^{1-1/\alpha}+1))^{1/\alpha}}^{(a_1,\ldots,a_k)}\setminus
\Omega_{(n-(a_1+\cdots+a_k)(n^{1-1/\alpha}+1))^{1/\alpha}}^{(a_1,\ldots,a_k)}.
$$
However,
\begin{align*}
&(n+(a_1+\cdots+a_k)(n^{1-1/\alpha}+1))^{k/\alpha}-(n-(a_1+\cdots+a_k)(n^{1-1/\alpha}+1))^{k/\alpha}\\
\leq&\frac{2k}{\alpha}(a_1+\cdots+a_k)(n^{1-1/\alpha}+1)\cdot(n+(a_1+\cdots+a_k)(n^{1-1/\alpha}+1))^{k/\alpha-1}\ll n^{(k-1)/\alpha}.
\end{align*}
\end{proof}
Let
$$
I=\{(i_1,\ldots,i_k):\,i_1>\cdots>i_k,\ i_1^\alpha+\ldots+i_k^\alpha<n<(i_1+1)^\alpha+\ldots+(i_k+1)^\alpha\}.
$$
By Lemma \ref{l1}, we have
\begin{align*}|\{(i_1,\ldots,i_{k-1}):\,i_1^\alpha+\ldots+2i_{k-1}^\alpha<n<(i_1+1)^\alpha+\ldots+2(i_{k+1}+1)^\alpha\}|\ll n^{(k-2)/\alpha}.
\end{align*}
It follows that
$$
\sigma_n-{\rm vol}(\Omega)=k!\sum_{(i_1,\ldots,i_k)\in I}(\delta_{i_1,\ldots,i_k}-d_{i_1,\ldots,i_k})+O(n^{(k-2)/\alpha}).
$$
For $i_1>\cdots>i_k$, since $\alpha_{i_1},\ldots,\alpha_{i_k}$ are independent, it is evident that
$$
\E\bigg(\sum_{(i_1,\ldots,i_k)\in I}(\delta_{i_1,\ldots,i_k}-d_{i_1,\ldots,i_k})\bigg)=0.
$$
\begin{Lem}[{\cite[Theorem 1.2]{H63}}]\label{hoe}
Let $\xi_1,\ldots,\xi_k$ be independent bounded real random variables. Suppose that $a_i\leq \xi_i\leq b_i$ and
$$
\sum_{i=1}^k(b_i-a_i)^2\leq D.
$$
Then for every $y$,
$$
\P(\eta-\E(\eta)\geq yD)\leq\exp(-2y^2),
$$
where $\eta=\xi_1+\cdots+\xi_k$.
\end{Lem}
Unfortunately, those $\delta_{i_1,\ldots,i_k}$ with
$(i_1,\ldots,i_k)\in I$ are not independent. In order to apply
Lemma \ref{hoe}, our main difficulty is to give a suitable
partition of $I$. The following lemma is the key of our proof of
Theorem \ref{t1}. And its proof will be given in the next section.
\begin{Lem}\label{ml} Let
$$
I_*=\{(i_1,\ldots,i_{k-1}):\,(i_1,\ldots,i_{k-1},i_k)\in I\text{ for some }i_k\},
$$
and for $(i_1,\ldots,i_{k-1})\in I_*$, let
$$
I_{i_1,\ldots,i_{k-1}}=\{i_k:\,(i_1,\ldots,i_{k-1},i_k)\in I\}.
$$
Then there exists a partition $I_*=U_1\cup U_2\cup\cdots\cup U_s$
satisfying that\medskip

\noindent(1) $s=O(n^{(k-2)/\alpha})$.\medskip

\noindent(2) $|U_t|=O(n^{1/\alpha})$ for $1\leq t\leq s$.\medskip

\noindent(3) For any $1\leq t\leq s$ and distinct $(i_1,\ldots,i_{k-1}),(j_1,\ldots,j_{k-1})\in U_t$, we have $I_{i_1,\ldots,i_{k-1}}\cap I_{j_1,\ldots,j_{k-1}}=\emptyset$ and $i_r\not=j_r$ for every $1\leq r\leq k-1$.

\noindent(4) For any $1\leq t\leq s$,
$$
\sum_{(i_1,\ldots,i_{k-1})\in U_{t}}|I_{i_1,\ldots,i_{k-1}}|^2=\begin{cases}
O(n^{2(\alpha-1)/\alpha^2}),&\text{if }\alpha>2,\\
O(n^{1/\alpha}),&\text{if }\alpha<2,\\
O(n^{1/2}\log n),&\text{if }\alpha=2.
\end{cases}
$$
\end{Lem}
Let's see how Theorem \ref{t1} can be deduced from Lemma \ref{ml}. For $(i_1,\ldots,i_{k-1})\in I_*$, let
$$
\xi_{i_1,\ldots,i_{k-1}}=\sum_{i_k\in I_{i_1,\ldots,i_{k-1}}}\delta_{i_1,\ldots,i_{k-1},i_k}.
$$
Clearly the possible values of $\xi_{i_1,\ldots,i_{k-1}}$ lie between $0$ and $|I_{i_1,\ldots,i_{k-1}}|$.
For $1\leq t\leq s$, define
$$
\eta_t=\sum_{(i_1,\ldots,i_{k-1})\in U_t}\xi_{i_1,\ldots,i_{k-1}}.
$$
In view of (3) of Lemma \ref{ml}, for distinct
$(i_1,\ldots,i_{k-1}),(j_1,\ldots,j_{k-1})\in U_t$,
$\xi_{i_1,\ldots,i_{k-1}}$ and $\xi_{j_1,\ldots,j_{k-1}}$ are
independent. And by (4) of Lemma \ref{ml} we have
$$
\sqrt{\sum_{(i_1,\ldots,i_{k-1})\in U_t}|I_{i_1,\ldots,i_{k-1}}|^2}\leq D=\begin{cases}C_1n^{(\alpha-1)/\alpha^2},&\text{if }\alpha>2,\\
C_1n^{1/(2\alpha)},&\text{if }\alpha<2,\\
C_1n^{1/4}\sqrt{\log n},&\text{if }\alpha=2,
\end{cases}
$$
for some constant $C_1$. Applying Lemma \ref{hoe} with
$y=\sqrt{((k-2)/\alpha+2)\log n}$, we can obtain
$$
\P(|\sigma_n-\E(\sigma_n)|\geq syD)\leq\sum_{t=1}^s\P(|\eta_t-\E(\eta_t)|\geq yD)\leq2s\exp(-2y^2)=\frac{O(n^{(k-2)/\alpha})}{n^{2(k-2)/\alpha+4}}\leq\frac{1}{n^2},
$$
for sufficiently large $n$. Then with the help of the
Borel-Cantelli lemma, we have
$$
\P(|\sigma_n-\E(\sigma_n)|\geq syD\text{ for infinitely many }n)=0,
$$
i.e., we almost surely have
$$
\sigma_n=\E(\sigma_n)+O(n^{(k-2)/\alpha}\sqrt{\log n}D)
$$
for sufficiently large $n$.

\section{Proof of Lemma \ref{ml}}
\setcounter{equation}{0} \setcounter{Thm}{0} \setcounter{Lem}{0}
\setcounter{Cor}{0}

In this section we shall prove Lemma \ref{ml}.
For $t\geq 0$, define
$$
X_{t}=\{(i_1,\ldots,i_{k-1})\in I_*:\, 4kt\alpha n^{1-1/\alpha}\leq i_1^\alpha+\cdots+i_{k-1}^\alpha<(4kt+2k)\alpha n^{1-1/\alpha}\}
$$
and
$$
X_{t}'=\{(i_1,\ldots,i_{k-1})\in I_*:\, (4kt+2k)\alpha n^{1-1/\alpha}\leq i_1^\alpha+\cdots+i_{k-1}^\alpha<4k(t+1)\alpha n^{1-1/\alpha}\}.
$$
Suppose that $(i_{1},\ldots,i_{k-1})\in X_{s}$ and $(j_{1},\ldots,j_{k-1})\in X_{t}$ where $s\not=t$. We claim that $I_{i_{1},\ldots,i_{k-1}}\cap I_{j_{1},\ldots,j_{k-1}}=\emptyset$. Assume on the contrary that
$u\in I_{i_{1},\ldots,i_{k-1}}\cap I_{j_{1},\ldots,j_{k-1}}$. Without loss of generality, suppose that $s<t$. Then we have
\begin{align*}
j_{1}^\alpha+\cdots+j_{k-1}^\alpha+u^\alpha\leq n\leq&(i_{1}+1)^\alpha+\cdots+(i_{k-1}+1)^\alpha+(u+1)^\alpha.
\end{align*}
Since $i_1,\ldots,i_{k-1},u\leq n^{1/\alpha}$ and $(x+1)^\alpha-x^\alpha\leq \alpha(x+1)^{\alpha-1}$ for $x\geq 0$, we get
\begin{align*}
&j_{1}^\alpha+\cdots+j_{k-1}^\alpha+u^\alpha\\
\leq&i_{1}^\alpha+\cdots+i_{k-1}^\alpha+u^\alpha+\alpha((i_{1}+1)^{\alpha-1}+\cdots+(i_{k-1}+1)^{\alpha-1}+(u+1)^{\alpha-1})\\
\leq&i_{1}^\alpha+\cdots+i_{k-1}^\alpha+u^\alpha+k\alpha(n^{1/\alpha}+1)^{\alpha-1}.
\end{align*}
On the other hand, by the definition of $X_t$, we have
\begin{align*}i_{1}^\alpha+\cdots+i_{k-1}^\alpha<&(4ks+2k)\alpha n^{1-1/\alpha}\\
\leq&(4kt-2k)\alpha n^{1-1/\alpha}\leq j_{1}^\alpha+\cdots+j_{k-1}^\alpha-2k\alpha n^{1-1/\alpha}.\end{align*}
This evidently leads an contradiction.
Similarly, for $s\not=t$, if $(i_{1},\ldots,i_{k-1})\in X_{s}'$ and $(j_{1},\ldots,j_{k-1})\in X_{t}'$, we also have $I_{i_{1},\ldots,i_{k-1}}\cap I_{j_{1},\ldots,j_{k-1}}=\emptyset$.

Let $d=\lfloor 2k^2\alpha\rfloor+k$. For $s_1,s_2,\ldots,s_{k-2}$, define
$$
Y_{s_1,\ldots,s_{k-2}}=\{(i_1,\ldots,i_{k-1})\in I_*:\,i_1+di_{v}=s_{v-1}\text{ for each }2\leq v\leq k-1\}.
$$
Clearly for distinct $(i_{1},\ldots,i_{k-1}),(j_{1},\ldots,j_{k-1})\in Y_{s_1,\ldots,s_{k-2}}$, we must have $i_{v}\not=j_{v}$ for all $v$. Below we shall show that $|Y_{s_1,\ldots,s_{k-2}}\cap X_t|\leq 1$ for arbitrary $s_1,\ldots,s_{k-2},t$. Assume that there exist distinct $(i_1,\ldots,i_{k-1}),(j_{1},\ldots,j_{k-1 })\in Y_{s_1,\ldots,s_{k-2}}\cap X_t$. Without loss of generality, suppose that $i_1<j_1$. For all $v$, since $i_1+di_v=j_1+dj_v$, we must have $j_1-i_1=dq$ and $i_v-j_v=q$ for some positive integer $q$. Hence
$$
i_{1}^\alpha+i_{2}^\alpha+\cdots+i_{k-1}^\alpha=(j_1-dq)^\alpha+(j_2+q)^\alpha+\cdots+(j_k+q)^\alpha.
$$
Clearly,
$$
j_1^\alpha-i_1^\alpha=j_1^\alpha-(j_1-dq)^\alpha\geq dq\cdot\alpha(j_1-dq)^{\alpha-1}=dq\alpha i_1^{\alpha-1},
$$
and for $2\leq v\leq k-1$,
$$
i_v^\alpha-j_v^\alpha=(j_v+q)^\alpha-j_v^\alpha\leq q\cdot\alpha(j_v+q)^{\alpha-1}=q\alpha i_v^{\alpha-1}.
$$
Recalling $i_1>i_2>\cdots>i_{k-1}$, we have $i_1>(n/k)^{1/\alpha}$. Thus
\begin{align*}
&(j_{1}^\alpha+j_{2}^\alpha+\cdots+j_{k-1}^\alpha)-(i_{1}^\alpha+i_{2}^\alpha+\cdots+i_{k-1}^\alpha)\\\geq&
dq\alpha
i_1^{\alpha-1}-q\alpha(i_2^{\alpha-1}+\cdots+i_{k-1}^{\alpha-1})\geq(d-k+1)q\alpha
i_1^{\alpha-1}>2k\alpha n^{1-1/\alpha}.
\end{align*}
It is impossible since both $(i_1,\ldots,i_{k-1})$ and $(j_{1},\ldots,j_{k-1 })$ lie in $X_t$. Similarly, $|Y_{s_1,\ldots,s_{k-2}}\cap X_t'|\leq 1$ for arbitrary $s_1,\ldots,s_{k-2},t$.

Now let
$$
U_{s_1,\ldots,s_{k-2}}=\bigcup_t \bigg(Y_{s_1,\ldots,s_{k-2}}\cap X_t\bigg)
$$
and
$$
U_{s_1,\ldots,s_{k-2}}'=\bigcup_t\bigg(Y_{s_1,\ldots,s_{k-2}}\cap X_t'\bigg).
$$
Below we shall only verify the requirements (1)-(4) for
$U_{s_1,\ldots,s_{k-2}}$, since all are similar to
$U_{s_1,\ldots,s_{k-2}}'$.

Note that $X_t\neq\emptyset$ implies that $$2kt\alpha n^{1-1/\alpha}\leq i_1^{\alpha}+\cdots+i_{k-1}^{\alpha}<n.$$
Thus clearly
$$
|\{t: X_t\neq\emptyset\}|=O(n^{1/\alpha}),
$$
i.e., $|U_{s_1,\ldots,s_{k-2}}|=O(n^{1/\alpha})$. Furthermore, since all those $s_1,\ldots,s_{k-2}$ are less than
$$
(d+1)i_1\leq(d+1)n^{1/\alpha},
$$
we have
$$|\{(s_1,\ldots,s_{k-2}):\,Y_{s_1,\ldots,s_{k-2}}\not=\emptyset\}|=O(n^{(k-2)/\alpha}).$$ Hence the number of $U_{s_1,\ldots,s_{k-2}}\not=\emptyset$ is $O(n^{(k-2)/\alpha})$.

Clearly,
$$
|I_{i_1,\ldots,i_{k-1}}|\leq(n-i_1^\alpha-\cdots-i_{k-1}^\alpha)^{1/\alpha}-(\min\{n-(i_1+1)^\alpha-\cdots-(i_{k-1}+1)^\alpha,0\})^{1/\alpha}+1.
$$
If $(i_1+1)^\alpha+\cdots+(i_{k-1}+1)^\alpha<n-n^{(\alpha-1)/\alpha}$, then
\begin{align*}
&(n-i_1^\alpha-\cdots-i_{k-1}^\alpha)^{1/\alpha}-(n-(i_1+1)^\alpha-\cdots-(i_{k-1}+1)^\alpha)^{1/\alpha}+1\\
\ll&
\frac{(i_1+1)^\alpha+\cdots+(i_{k-1}+1)^\alpha-i_1^\alpha-\cdots-i_{k-1}^\alpha}{(n-(i_1+1)^\alpha-\cdots-(i_{k-1}+1)^\alpha)^{1-1/\alpha}}\\
\ll&\frac{(i_1+1)^{\alpha-1}+\cdots+(i_{k-1}+1)^{\alpha-1}}{(n-(i_1+1)^\alpha-\cdots-(i_{k-1}+1)^\alpha)^{1-1/\alpha}}.
\end{align*}
Note that for any $U_{s_1,\ldots,s_{k-2}}$ and $(i_1,\ldots,i_{k-1}),(j_1,\ldots,j_{k-1})\in U_{s_1,\ldots,s_{k-2}}$, we have
$$
|i_1^\alpha+\cdots+i_{k-1}^\alpha-(j_1^\alpha+\cdots+j_{k-1}^\alpha)|\geq 4k\alpha n^{1-1/\alpha}.
$$
Since
$$
(i_1+1)^\alpha+\cdots+(i_{k-1}+1)^\alpha-i_1^\alpha-\cdots-i_{k-1}^\alpha\ll n^{(\alpha-1)/\alpha},
$$
the number of $(i_1,\ldots,i_{k-1})\in U_{s_1,\ldots,s_{k-2}}$ satisfying
$$
(i_1+1)^\alpha+\cdots+(i_{k-1}+1)^\alpha\geq n-n^{(\alpha-1)/\alpha}.
$$
is at most $O(1)$. Thus
\begin{align*}
&\sum_{(i_1,\ldots,i_{k-1})\in U_{s_1,\ldots,s_{k-2}}}|I_{i_1,\ldots,i_{k-1}}|^2\\
=&
\sum_{\substack{(i_1,\ldots,i_{k-1})\in U_{s_1,\ldots,s_{k-2}}\\
(i_1+1)^\alpha+\cdots+(i_{k-1}+1)^\alpha\geq
n-n^{(\alpha-1)/\alpha}}}|I_{i_1,\ldots,i_{k-1}}|^2
+\sum_{\substack{(i_1,\ldots,i_{k-1})\in U_{s_1,\ldots,s_{k-2}}\\
(i_1+1)^\alpha+\cdots+(i_{k-1}+1)^\alpha<n-n^{(\alpha-1)/\alpha}}}|I_{i_1,\ldots,i_{k-1}}|^2\\
\ll&(n^{(\alpha-1)/\alpha^2})^2+\int_{0}^{n^{1/\alpha}-1}\bigg(\frac{n^{1-1/\alpha}}{(n-t n^{1-1/\alpha})^{1-1/\alpha}}\bigg)^2d t.\end{align*}
That is,
$$
\sum_{(i_1,\ldots,i_{k-1})\in U_{s_1,\ldots,s_{k-2}}}|I_{i_1,\ldots,i_{k-1}}|^2=\begin{cases}
O(n^{2(\alpha-1)/\alpha^2}),&\text{if }\alpha>2,\\
O(n^{1/\alpha}),&\text{if }\alpha<2,\\
O(n^{1/2}\log n),&\text{if }\alpha=2.
\end{cases}
$$\qed

\end{document}